\documentclass[a4paper,11pt]{amsart}

\usepackage{mathrsfs}
\usepackage[all]{xy}
\usepackage[latin1]{inputenc}

\newtheorem{theo}{Theorem}
\newtheorem{lemm}[theo]{Lemma}
\newtheorem{prop}[theo]{Proposition}
\newtheorem{conj}[theo]{Conjecture}

\theoremstyle{remark}
\newtheorem{rema}[theo]{Remark}

\let\eps\varepsilon
\let\epsilon\varepsilon
\let\phi\varphi
\let\ra\rightarrow
\def\pr{\operatorname{pr}}
\def\Q{\mathbf Q}
\def\R{\mathbf R}
\def\C{\mathbf C}
\def\Ric{\operatorname{Ric}}

\def\norm#1{\mathopen\|#1\mathclose\|}
\def\ddc{\mathop{\mathrm d\mathrm d^{\mathrm c}}}

\def\GL{\operatorname{GL}}

\def\SL{\operatorname{SL}}
\def\SO{\operatorname{SO}}
\def\MAT{\mathrm M}
\def\PH{\mathfrak h}

\begin{document}

\title[On the canonical degrees of curves]{On the canonical degrees of curves in varieties of general type}
\author{Pascal Autissier}
\address{IMB, Universit\'e Bordeaux~1, 
351 cours de la Lib\'eration, F-33405 Talence Cedex}
\email{pascal.autissier@math.u-bordeaux1.fr}

\author{Antoine Chambert-Loir}
\address{Irmar, Universit\'e de Rennes~1, Campus de Beaulieu, F-35042 Rennes Cedex, France}
\email{antoine.chambert-loir@univ-rennes1.fr}

\author{Carlo Gasbarri}
\address{IRMA \\ 
Universit\'e de Strasbourg \\
7, Rue Ren\'e Descartes \\
F-67000 Strasbourg }
\email{gasbarri@math.unistra.fr}
\maketitle

\let\subsection\section

\section*{Introduction}

In this paper, we work in the framework of complex analytic varieties;
without contrary mention, varieties are assumed to be irreducible (and reduced).
If  $C$ is a projective curve, we let
$g_C$ be its \emph{geometric} genus (namely, the genus
of its desingularization) and $\chi(C)=2-2g_C$ its geometric
Euler characteristic; we also write $\deg_C L$
for the degree of a line bundle~$L$ on~$C$.

Given a smooth projective variety of general type~$X$ 
and a integer~$g$, 
a widely believed conjecture predicts that
there exists a strict closed subset~$Z$ of~$X$ and an integer~$d$
such that every curve of genus~$g$ drawn on~$X$ and not contained in~$Z$ 
has degree at most~$d$.
An effective version of this conjecture can be stated in
the following way:

\begin{conj}\label{conj} 
Let $X$ be a smooth projective variety of general type;
let $K_X$ be the canonical line bundle of $X$.
Then there exist real numbers~$A$ and~$B$, and a strict Zariski closed
subset $Z\subset X$ such that for every
curve $C\subset X$ such that $C\not\subset Z$, 
we have
\[ \deg_C(K_X)\leq A(2g_C-2)+B.\]
\end{conj}
Combined with the Riemann--Hurwitz formula,
this conjecture would imply that for 
any finite morphism $g\colon C\ra X$ such that $g(C)\not\subset Z$, 
\[ \deg_C(g^*K_X)\leq A(2g_C-2)+B [C:g(C)].\]

The conjecture is obvious when $\dim(X)=1$.
For surfaces, the most important result in its direction  is a now well-known 
theorem by Bogomolov~\cite{bogomolov1977} (and its further developments) 
who proves such a statement 
under the additional assumption that $X$
is a minimal surface with positive Segre class.
Under the supplementary assumption that $X$ is Brody-hyperbolic, 
\emph{i.e.,} that any holomorphic map from~$\C$ to~$X$ is constant,
Demailly established this conjecture
in~\cite{demailly1997} (Theorem~2.1); he shows that one may take
$Z=\emptyset$ and $B=0$.

Of course, an effective geometric upper-bound of the
real numbers~$A$ and~$B$ of the conjecture would be
of the utmost interest.
In his interesting paper~\cite{miyaoka2008}, Miyaoka
establishes a fundamental inequality between the self-intersection
of a projective curve~$C$ drawn on a minimal 
projective surface of general type~$X$,
its canonical degree and numerical invariants of the surface itself.
When one restricts oneself to smooth curves, one has $C^2=2g_C-2-\deg_C(K_X)$,
and Miyaoka can show effectively, for any
positive real number~$\eps$,
the existence of a real number~$B$ (depending on~$\eps$)
such that every \emph{smooth} projective curve~$C$ contained in~$X$ satisfies
\[ \deg_C K_X \leq({\frac{3}{2}}+\epsilon)(2g_C-2)+B. \]
This analysis can be extended to nodal curves with a bounded number
of singular points.

While no method is known to understand the constant~$B$,
many interesting conjectures and
guesses have been made concerning the constant~$A$ and its geometric meaning. 

Exploiting the analogy between number theory, Nevanlinna theory
and algebraic geometry, P. Vojta~\cite{vojta1987} proposed very
deep and ambitious conjectures bounding
the canonical height of algebraic points of bounded degree in fibrations. 
The first version of his
conjectures suggests that for any positive real
number~$\eps$, conjecture~\ref{conj} could hold with $A=\dim(X)+\epsilon$, 
the constant~$B$ and the exceptional locus~$Z$ depending on~$\epsilon$,
provided one restricts oneself to curves of given gonality. 
(Recall that the gonality of a projective
curve is the least degree of a non-constant
morphism to the projective line; obviously, 
horizontal curves of bounded degree in a fibration have bounded gonality.)
A more recent
version of his conjecture, published in~\cite{vojta1998}
replaces~$A=\dim X+\eps$ by the more optimistic~$A=1+\eps$,
again with the same restriction on gonality.

More prudently, McQuillan~\cite{mcquillan2001b} asked whether
the constant~$A$ could be
related to the holomorphic sectional curvature of the variety~$X$.
The holomorphic sectional curvature of hermitian symmetric domains
has been explicitly bounded by Azukawa in~\cite{azukawa1985}, 
see Lemma~\ref{lemm.azukawa} below.
Our first contribution in this paper is the observation that 
this result 
implies a precise and explicit form  of conjecture~\ref{conj}
for compact locally symmetric varieties, with $A=\dim(X)$.
\begin{prop}\label{theo.azukawa}
Let $M$ be a bounded hermitian domain
and let $X$ be a compact quotient of~$M$.
Then, for any projective curve~$C$ 
and any finite morphism $g\colon C\ra X$,
\[ \deg_C(g^*K_X) \leq \dim(X) (2g_C-2). \]
\end{prop}

Conversely, and this was the original motivation of this note,
we also want to indicate, by very elementary
examples, that one cannot go under $A=\dim(X)$ in general,
or, in other words, that Vojta's refined expectation cannot hold without
the above-mentioned restriction on the gonality.
Namely, we shall prove the following  theorem.

\begin{theo}\label{theo}
Let $d$ be an integer such that $d\geq 2$.
There exists a smooth projective
variety~$X$ of dimension~$d$ whose canonical
bundle~$K_X$ is ample and a family $(C_n,f_n)$, where $C_n$
is a smooth projective curve 
and $f_n\colon C_n\ra X$ is a finite morphism satisfying the
following properties:
\begin{enumerate}\def\theenumi{\alph{enumi}}\def\labelenumi{\theenumi)}
\item
$\deg_{C_n}(f_n^*K_X) = d( 2g_{C_n}-2 )$;
\item
$\displaystyle\lim_n g_{C_n}=+\infty$;
\item
the union of their images $f_n(C_n)$ is dense in~$X$
for the Zariski topology.
\end{enumerate}
\end{theo}

We show (Theorem~\ref{theo.shimura}) that 
one may take for $X$ any compact Shimura variety
associated to any $\Q$-anisotropic, semisimple, 
simply connected $\Q$-group~$\mathbf G$,
provided $X$ contains at least one totally geodesic compact Riemann surface~$C$.
The curves $C_n$ are then images of~$C$ under Hecke correspondences.
Using weak approximation at the real place, we show indeed 
that these curves are Zariski dense in~$X$.

We give in a final section an entirely algebraic
approach to this Theorem, using powers of Shimura curves.
Because of the cusps, 
considering modular curves in a power of a classical
modular curves leads to a slightly weaker result.

We view these examples as a nice implementation of McQuillan's principle
that \emph{Shimura varieties are a rich source of examples
for Diophantine geometry.}

\medskip

A first observation to make is that Theorem~\ref{theo} does not contradict 
Miyaoka's result~\cite{miyaoka2008}
mentioned above. Indeed, there is no reason that the curves
$f_n(C_n)$ in~$X$ are smooth. 
However, in our examples, the morphism $f_n$ will be an immersion
in the sense of differential geometry, \emph{i.e.},
its tangent map is injective at each point.

Observe also that Theorem~\ref{theo} does not contradict Vojta's expectations 
either,
since he restricts himself to curves in a fibration
which have bounded degree. The gonality of such curves is obviously
bounded, while, according to a theorem of Abramovich~\cite{abramovich1996},
the gonality of modular curves tends to infinity.

The geometric conjecture~\ref{conj} has an arithmetic
analogue, due to Vojta.
In a similar spirit, we refer to the paper~\cite{levin2011}
for examples that show
that Vojta's conjecture  cannot be extended to algebraic
points.

\bigskip

\noindent\emph{Acknowledgments:}
This work began during an instructional conference
on Diophantine Geometry, organized by two of us in University
of Rennes under the auspices
of Centre Émile Borel and Institut Henri Poincar\'e.
We thank these organizations to have given us the opportunity
to discuss this subject.

We also thank B.~Edixhoven, A.~Levin, M.~McQuillan,
E.~Ullmo and P.~Vojta for their interest 
and their comments.

A.C-L. acknowledges the support of Institut universitaire
de France, as well as that of the National Science
Foundation under agreement No.~DMS-0635607
during his stay at the Institute of Advanced Study (Princeton),
where part of the  work presented here has been done.
The three authors are supported by the project
\emph{Positive} of Agence nationale de la recherche,
under agreement ANR-2010-BLAN-0119-01.

\label{sec:shimura2}

\subsection{Hermitian symmetric spaces; Bergman metric}

We first review well known general
facts of differential geometry in Hermitian symmetric spaces. 
Since different authors use different normalizations, we found
necessary to recall the precise definitions.

Let $M$ be a Hermitian symmetric space, namely
a connected K\"ahler complex manifold
such that each point of~$M$ is an isolated fixed point
of an involutive holomorphic isometry of~$M$. Set $m=\dim(M)$.

Following the notation of~\cite{helgason1978} (chap.~VIII, \S4),
let $A(M)$ be the group of holomorphic isometries of~$M$
and let $G=A_0(M)$ be its identity component; this is
a real connected Lie group.
We fix a point~$o\in M$ and consider its stabilizer~$K$ in~$A_0(M)$;
this is a compact subgroup; as a real manifold, $M$ can be identified
to the homogeneous space~$G/K$.

We shall always assume that $M$ is of noncompact type:
$G$ is semisimple and $K$ is a maximal compact subgroup of~$G$;
then $M$ is simply connected (\cite{helgason1978}, p.~376, Theorem~4.6)
and there exists (\cite{helgason1978}, p.~382, Theorem~7.1) 
a holomorphic diffeomorphism from~$M$
to a bounded symmetric domain~$D_M$ in the complex space~$\C^m$,
``symmetric'' meaning that each point of~$D_M$ is an isolated
fixed point of an involutive holomorphic diffeomorphism
of~$D_M$ onto itself.

Over the tangent bundle $TD_M$ of~$D_M$ there exists a  K\"ahler
metric $\omega$, for which the following property holds: Let
$\norm{\cdot}_{\mathrm KE}$ the induced metric on the canonical
bundle $K_{D_M}=\bigwedge^m\Omega^1_{D_M}$, then the following
formula relates the K\"ahler form $\omega$ and the first Chern
form of $K_{D_M}$:
\begin{equation}\label{eqn.chern-M}
\omega = c_1(K_{D_M},\norm\cdot_{\mathrm KE}).
\end{equation}
Such a metric is called K\"ahler-Einstein and in this case it coincides with the Bergman metric. Its  existence is proved for instance in Proposition~3.6 (p.~371)  of~\cite{helgason1978}. Once we fix such a normalization, 
as we do now, this metric is uniquely determined. 
Any holomorphic diffeomorphism of~$D_M$ is an isometry
with respect to that metric (\cite{helgason1978}, p.~370, Prop.~3.5). 
We assume, as we may, that the K\"ahler structure of~$M$ is 
the pull-back of the fixed K\"ahler-Einstein metric on~$D_M$
by some (hence, any) holomorphic diffeomorphism of~$M$ to~$D_M$.

\subsection{Holomorphic sectional curvature}
Let us recall a definition of the holomorphic sectional curvature.
Let $p\in M$ and let us consider holomorphic maps
$\gamma\colon U\ra M$, where $U$ is the  open unit disk
in~$\C$, such that $\gamma(0)=0$ and $\gamma'(0)\neq 0$.
Let us write $\gamma^* \omega= F_\gamma(z) i \mathrm dz\wedge\mathrm d\bar z$
on~$U$, where $F_\gamma$ is nonnegative, $F_\gamma(0)\neq 0$.
In a neighborhood of~$0$, the form $\gamma ^*\omega$ defines a metric on the canonical bundle~$K_U$ and its curvature form is given by
\[ R_\gamma = \ddc \log(F_\gamma). \]
The Ricci curvature $\Ric_\gamma$ of $\gamma^*\omega$ is
defined by the formula
\[ \Ric_\gamma(z) F_\gamma(z) i \mathrm dz\wedge\mathrm d\bar z
 =  - \ddc \log (F_\gamma)  \]
where $\ddc = \dfrac i{2\pi}\partial\bar\partial$.
This is a function on a neighborhood of~$0$ in~$U$,
and its value at~$z=0$ only depends on the tangent line at~$p$ 
generated by~$\gamma'(0)$. 

The holomorphic sectional curvature $S_M(p)$ at~$p$
is then defined as 
\begin{equation}
S_M(p) = \sup_{\substack{\gamma\colon U\ra M\\ \gamma(0)=p}}
     \Ric_\gamma(0)
\end{equation}

Since the group of holomorphic isometries of~$M$
acts transitively on~$M$, the holomorphic sectional curvature
is actually independent of~$p$; we shall denote it by~$S_M$.
Since $M$ is a Hermitian symmetric space of the non-compact
type (\cite{helgason1978}, Ch.~VIII, Theorem~7.1, p.~382),
we know that $S_M$ is negative (\cite{helgason1978},  Ch.~V,
Theorem~3.1, p.~241). 

Corollary 8.5 (i) in~\cite{azukawa1985} asserts
that the holomorphic sectional curvature satisfies
the following inequality:
\begin{lemm}[Azukawa] \label{lemm.azukawa}
One has
\[ \frac{-1}{S_M} \leq \dim(M).\]
\end{lemm}

\subsection{Compact quotients}
Let $\Gamma$ be a cocompact discrete subgroup of~$G$ which
acts without fixed points on~$M$. Then, $X_\Gamma=\Gamma\backslash M$
is a compact connected complex manifold. The K\"ahler-Einstein structure
descends to a K\"ahler-Einstein metric~$\omega_\Gamma$ on~$X_\Gamma$.
The canonical line bundle~$K_{X_\Gamma}$ is naturally metrized;
in view of Equation~\eqref{eqn.chern-M},
its Chern form satisfies 
\[ \omega_{\Gamma} = c_1(K_{X_\Gamma},\norm\cdot), \]
hence is positive. By Kodaira's theorem, ${X_\Gamma}$
is a projective complex manifold.

\subsection{Compact Riemann surfaces}
Let $Y$ be a smooth compact Riemann surface and let $f\colon Y\ra X_\Gamma$
be a finite morphism.
One has
\begin{equation}
\int_Y f^*(\omega_\Gamma) 
    = \deg_{Y}(f^*K_{X_\Gamma}).
\end{equation}
Moreover, the form $f^*(\omega_\Gamma)$ induces a
hermitian metric on the canonical line bundle~$K_Y(-\Delta_f)$ of~$Y$,
where $\Delta_f$ is the branching divisor of~$f$.
Let $R_Y$ be its curvature,
one has
\[ \int_Y R_Y = \deg(K_Y-\Delta_f)= -\chi(Y)-\deg(\Delta_f). \]
By the definition of the holomorphic sectional curvature,
\[ R_Y\geq -S_M  f^*(\omega_\Gamma). \]
This shows that
\begin{equation}\label{ineq.chi-deg}
-\chi(Y)\geq - \chi(Y) - \deg(\Delta_f)
     \geq -S_M \deg_{Y}(f^*K_{X_\Gamma}).
\end{equation}
Given Lemma~\ref{lemm.azukawa}, this establishes Proposition~\ref{theo.azukawa}.

\subsection{Totally geodesic Riemann surfaces}
Let $(X,\omega)$ be a Kähler complex manifold, let $Y$ be a Riemann  surface
and let $f\colon Y\ra X$ be a holomorphic immersion. We say
that $f$ is totally geodesic  if any geodesic of~$X$ contained
in~$Y$ is a geodesic of~$Y$ with respect to the Kähler form $f^*\omega$.
This is equivalent to the vanishing of the second fundamental form of~$Y$
(\cite{helgason1978}, Theorem~14.5, p.~80)
and implies that for each $p\in Y$,
the Ricci curvature of $f^*\omega$ at~$p$
equals the holomorphic sectional curvature of~$X$ at $f(p)$.
(This follows, for example, from the Gauss equation relating
the curvature tensor of a Riemannian manifold, that of
a submanifold and the second fundamental form,
see~\cite{besse:2008}, Theorem~1.72.)

Let us assume that $X=X_\Gamma=\Gamma\backslash M$.
Any immersion $f\colon Y\ra X$ lifts to a holomorphic
immersion $\tilde f\colon\mathbf D\ra M$, where
$\mathbf D$ is the open unit disk.
The conditions for $f$ and~$\tilde f$ to be totally geodesic 
are equivalent. Moreover, since automorphisms of~$M$ preserve
its Kähler form, we obtain that for any such  automorphism~$g$,
$g\circ\tilde f$ is totally geodesic if and only if $\tilde f$
is.

For a totally geodesic map $f\colon Y\ra X_\Gamma$, where $Y$ 
is a compact Riemann surface,
inequality~\eqref{ineq.chi-deg} becomes an equality:
\begin{equation}\label{eq.chi-deg}
- \chi(Y) = -S_M  \,\deg_{Y}(f^*K_{X_\Gamma})
\qquad \text{($Y$ totally geodesic).}
\end{equation}

\subsection{Constructing one totally geodesic Riemann surface}
\label{ss.quaternions}

It does not seem easy to construct totally geodesic compact
Riemann surfaces
in arbitrary compact quotients of symmetric hermitian domains.
However, let us show some examples, 
coming from the theory of Shimura varieties.

We restrict the situation to the case where the
Lie group~$G$ is given by $G=\mathbf G(\R)$,
for some connected semisimple and simply connected $\Q$-\emph{anisotropic}
linear algebraic group~$\mathbf G$ over~$\mathbf Q$.
Then, for any congruence subgroup $\Gamma$ 
which acts without fixed point on~$M=G/K$,
the quotient space $X_\Gamma=\Gamma\backslash M$ is a smooth projective variety.

Let $D_1$ be a quaternion algebra over~$\mathbf Q$
and let $\mathbf G_1=\SL(D_1)$ be the corresponding
almost simple $\Q$-group. We assume that $D_1$ is not split,
\emph{i.e.}, is not isomorphic to~$\MAT_2(\mathbf Q)$
but that $D_1\otimes_{\Q} \R$ is split,
so that $G_1=\mathbf G_1(\R)\simeq\SL_2(\R)$.
Let $K_1\simeq \SO_2(\R)$ be a maximal compact subgroup in~$G_1$, let
$M_1=G_1/K_1$; then $M_1$ is a hermitian symmetric domain,
and is isomorphic to the Poincar\'e upper half-plane.
Finally, let $\Gamma_1$ be a congruence subgroup in~$G_1(\Q)$,
small enough so as to acts freely on~$M_1$.
Then $X_1=\Gamma_1\backslash M_1$ is a smooth curve.

Let $F$ be a totally real field such that
the simple $F$-algebra $D=D_1\otimes_{\mathbf Q} F$ is not split.
Let $\mathbf G=\SL(D)$ be the corresponding almost simple~$F$-group;
by Weil's restriction of scalars, we view it
as a semisimple~$\mathbf Q$-group.
One has $G=\mathbf G(\R)\simeq \SL(2,\R)^d$, where $d=[F:\Q]$.
Similarly, let $K$ be a maximal compact subgroup of~$G$,
let $M=G/K$ and let $\Gamma$ be a congruence subgroup
of~$\mathbf G(\Q)$ acting freely on~$M$.
Then, $X=\Gamma\backslash M$  is a smooth projective variety.

Moreover, the algebraic group~$\mathbf G_1$ embeds naturally in~$\mathbf G$;
over~$\R$, this is the diagonal embedding $\SL(2,\R)\hookrightarrow
\SL(2,\R)^d$. Let us assume that $K$ contains~$K_1$
and that $\Gamma$ contains~$\Gamma_1$.
Then, the immersion~$\mathbf G_1\hookrightarrow \mathbf G$
induces an immersion $X_1\hookrightarrow X$ whose image
is a totally geodesic Riemann surface.
(This follows, from example, from~\cite{helgason1978},
Theorem~7.2, p.~224. See also~\cite{moonen:1998}, §4.1.)

Let us finally remark that in this example,
$M$ is a power of the Poincar\'e disk, hence
its holomorphic sectional curvature is equal to $-1/\dim(M)$.

\subsection{Hecke operators}
Any $g\in G(\mathbf Q)$
gives rise to a \emph{Hecke correspondence}~$T_g$ on~$X_\Gamma$.
Let us recall its definition.
Let $\Gamma_g=\Gamma\cap g^{-1}\Gamma g$; this
is a congruence subgroup of finite index in~$\Gamma$.
The variety $\Gamma_g\backslash M$ is smooth
and admits two \'etale maps to~$X_\Gamma$, given respectively by
\[ p_1(\Gamma_g x)=\Gamma x, \qquad p_2(\Gamma_g x)=\Gamma g x, \]
for any $x\in M$, hence a correspondence on~$X_\Gamma$.

Let us describe explicitly the image of a subvariety~$Y\subset X_\Gamma$
by this correspondence.
Let $(h_1,\dots,h_r)$ be a family of elements of~$\Gamma$
which represent each left-class modulo~$\Gamma_g$.
Let $\tilde Y\subset M$ be the preimage of~$Y$; then,
\[ T_g([Y]) = \sum_{i=1}^r [p(g h_i \tilde Y)], \]
where $p\colon M\ra X_G$ is the natural projection.

In view of this description, we observe that if $Y$ is the totally geodesic
image in~$X_\Gamma$ of a Riemann surface,
then so are all the irreducible components of $T_g([Y])$.

\subsection{Weak approximation}
Since the group~$G$ is connected,
$\mathbf G(\Q)$ is dense in~$G=\mathbf G(\R)$
(\cite{platonov-rapinchuk:1994}, Theorem~7.7, p.~415).
Let $\tilde Y\subset G$ be the preimage of~$Y$ by the projection
map $G\ra X_\Gamma=\Gamma\backslash G/K$.
The union of the curves $a\tilde Y$, for $a\in G(\Q)$, is therefore
dense in~$G$. It follows that the union of the curves $T_a Y$
is dense in~$X_\Gamma$.

Let $Y_0$ be any fixed totally geodesic Riemann surface in~$X_\Gamma$.
We conclude that the irreducible components
of the curves $T_{a}Y_0$, for $a\in G(\Q)$, are dense in~$X_\Gamma$.
Since these components are also totally geodesic, we obtain 
the following result.

\begin{theo}\label{theo.shimura}
Let $X_\Gamma$ be a compact Shimura  variety associated
to a semisimple, connected and simply connected
$\Q$-anisotropic algebraic group over~$\Q$; let us assume that
$X_\Gamma$ contains a totally geodesic Riemann surface.

Then, the set of projective (algebraic) curves~$Y$ in~$X_\Gamma$ such that
\[ \deg _{K_{X_\Gamma}}(Y) = \left( \frac{-1}{S_{X_\Gamma}}\right)
   \left( 2g_Y-2  \right) \]
is dense in~$X_\Gamma$ for the complex topology;
in particular, it is Zariski dense.
\end{theo}

\begin{rema}
When it applies, a theorem
of Clozel, Oh and Ullmo~\cite{clozel-oh-ullmo2001}
refines the density of Hecke orbits into an equidistribution theorem.
\end{rema}

\subsection{Shimura curves}

In this section we give another, more explicit, approach to
Theorem~\ref{theo}. It is obtained by considering products of Shimura
curves and Hecke correspondences over them. Strictly speaking this
section can be viewed as a particular case of the previous one 
but the present approach is entirely algebraic and relies
on  the modular interpretation of the Shimura curves. 
A similar analysis can be  applied 
to the graph of Hecke correspondences on modular curves;
because of the cusps,
this leads to a slightly weaker form of Theorem~\ref{theo}. 

Our main references concerning Shimura curves and their
moduli interpretation
are~\cite{miyake2006,buzzard1997};
see also~\cite{vigneras80,boutot-carayol1991}.

Let $D$ be an indefinite quaternion algebra over~$\Q$
of reduced discriminant~$\delta>1$, viewed as
a subalgebra of~$\MAT_2(\R)$; let $R$ be a fixed
maximal order in~$D$ (they are all conjugate)
and let $\Gamma_R$ be the subgroup $R^*\cap\SL_2(\R)$ of~$\GL_2(\R)$.
Elements of~$\Gamma_R$ act by homographies on the Poincar\'e upper-half-plane~$\PH$.
The Shimura curve~$C$ associated to~$D$ is the quotient of~$\PH$
by~$\Gamma_R$. 
It is a compact algebraic complex curve
whose points are in one-to-one correspondence with
the set of (isomorphism classes of)  complex abelian surfaces 
endowed with an action of~$R$, also called \emph{false elliptic
curves}. 
These abelian surfaces carry a canonical  principal polarization
once an element of square~$-\delta$ is fixed in~$R$. 

See also~\cite[p.~594]{buzzard1997}
for an adelic description of this Shimura curve, 
in the spirit of the general theory
of Shimura varieties, as well as for its coverings induced by level-structures
on these abelian surfaces.
In particular,
for any square free positive integer~$N$ which is prime to~$\delta$,
there exists a projective curve~$C_N$ parameterizing
isogenies of \emph{false degree}~$N$
between \emph{false elliptic curves}, that is isogenies $f\colon A\ra A'$
commuting with the action of~$R$ such that the composition $f^\vee\circ f$
of~$f$ with its dual isogeny~$f^\vee$ is the multiplication by~$N$.
There are natural maps $s_M^N\colon C_N\ra C_M$ 
and $t_M^N\colon C_N\ra C_M$ between these curves,
whenever $M$ divides~$N$. 
If $N\geq 4$, this curve is a \emph{fine} moduli space;
equivalently, $C_N(\C)$ is the quotient of the upper-half-plane
by a discrete, cocompact, and torsion free 
subgroup of $\SL_2(\R)$.
Since there are no cusps at infinity,
the complex description of these curves,
or their moduli interpretation,
shows that these maps are finite and \emph{\'etale}.

Let $p>4$ be a fixed prime number, let $d\geq 2$ be a positive
integer and let $X$ be the smooth variety~$C_p^d$.
Let $N=p\ell_1\dots\ell_{d-1}$ be the product of~$d$ distinct prime numbers
not dividing~$\delta$, among which the prime~$p$.
The curve $C_N$ admits $d$ distinct morphisms $g_1,\dots,g_d$
to the curve~$C_p$, given by
$g_d=s_p^N$ and $g_i=t_p^{p\ell_i }\circ s_{p\ell_i}^N$ for $1\leq i\leq d-1$.
Let $f_N\colon C_N\ra X$ be the morphism $(g_1,\dots,g_d)$.

Our claim is that the variety~$X$ and the curves
$(C_N,f_N)$ we obtain satisfy the statement of Theorem~\ref{theo}.
Since one has
\[ \deg f_N^* K_X=-d\chi(C_N), \]
for any integer~$N$,
this follows from the following density property.

\begin{lemm}\label{lemm.dense}
Let $p$ and~$d$ be fixed.
Then the union of the curves $f_N(C_N)$
is dense for the Zariski topology, where $N$ varies
among the set of integers of the form $N=\ell_1\dots\ell_{d-1}p$ as above.
\end{lemm}
\begin{proof}
The proof is by induction on~$d$. When $d=1$,
it clearly holds since $X=C_p$ and $f_p\colon C_p\ra X$
is the identity map.
Let us assume that the result holds for~$d-1$ and let
us prove it for~$d$. Assume by contradiction
that there exists a divisor~$H$ in~$X$
containing the images of all curves~$C_N$ above.

Let $X'=C_p^{d-1}$
and 
let  $q\colon X\ra X'$,
$(x_1,\dots,x_d)\mapsto (x_2,\dots,x_d)$,
 be the projection
obtained by forgetting the first component.
For any integer~$N'$ of
the form $p\ell_2\dots\ell_{d-1}$ which is the
product of distinct prime numbers,
let $f'_{N'}\colon C_{N'}\ra X'$
be the corresponding morphism. 
If $\ell_1$ is any prime number not dividing~$N'$ and $N=\ell_1N'$,
there is a commutative diagram
\[ \xymatrix{
    C_N   \ar[r]^{f_N} \ar[d]^{s_{N'}^N} & X \ar[d]^q \\
  C_{N'} \ar[r]^{f'_{N'}} & X' .}  \]
By hypothesis, all curves $f'_{N'}(C_{N'})$
are contained in~$q(H)$, hence $q(H)=X'$ by induction.
Since $\dim H=\dim X'$
and the restriction~$q_H$ of~$q$ to~$H$ is surjective, there
exists a point $y'\in C_{N'}$ such that $q_H$ is finite 
above $x'=f'_{N'}(y')$. 
This point~$y'$ corresponds to an isogeny $u\colon A\ra A'$
between abelian surfaces endowed with an action of~$R$, 
commuting with the action on~$R$,
such that $u^\vee\circ u$ is the multiplication by~$N'$.

For any prime number~$\ell$ which does not divide~$N'$
and which is totally decomposed in~$R$, there exists
a cyclic subgroup $K_\ell$ of order~$\ell$ in~$A$ which
is invariant by~$R$.
The subgroup $K_\ell+\ker(u)$ of~$A$ is then cyclic of order~$N=\ell N'$,
and $R$-invariant, so that
the point $f_N(A\ra A/(K_\ell+\ker(u)))$ of~$X$ is a point of~$H$
mapping to~$f'_{N'}(A\ra A')$ by~$q$. 
Then,
$\pr_1(f_N(A\ra A/(K_\ell+\ker(u))))=
g_1(A \ra A/(K_\ell+\ker(u)))=(A/K_\ell\ra A/(K_\ell+\ker(u)))$.

Moreveor, for any given prime number~$\ell$ which does not
divide~$\delta$, there are infinitely many prime
numbers~$\ell'$ such that 
our false elliptic curves have no isogeny of false degree~$\ell\ell'$,
for, otherwise, almost all prime numbers would be reduced norms in~$D$.
Consequently, the fibre of~$q_H$ above $f'_{N'}(A\ra A')$ is infinite, 
contradiction.
\end{proof}

\providecommand{\noopsort}[1]{}\providecommand{\url}[1]{\textit{#1}}
\providecommand{\bysame}{\leavevmode ---\ }
\providecommand{\og}{``}
\providecommand{\fg}{''}
\providecommand{\smfandname}{\&}
\providecommand{\smfedsname}{\'eds.}
\providecommand{\smfedname}{\'ed.}
\providecommand{\smfmastersthesisname}{M\'emoire}
\providecommand{\smfphdthesisname}{Th\`ese}

\end{document}